\documentclass[11pt]{article}
\usepackage{amssymb,amsfonts,amsmath,latexsym,epsf,tikz,url}
\newtheorem{theorem}{Theorem}[section]

\newtheorem{corollary}[theorem]{Corollary}
\newtheorem{lemma}[theorem]{Lemma}

\newtheorem{definition}[theorem]{Definition}

\newcommand{\qed}{\hfill $\square$\medskip}

\begin{document}
\title{On the number of connected edge cover sets in a graph }

\author{Mahsa Zare$^1$ \and Saeid Alikhani$^{1,}$\footnote{Corresponding author} \and Mohammad Reza Oboudi$^2$}

\date{\today}

\maketitle

\begin{center}

  $^1$Department of Mathematical Sciences, Yazd University, 89195-741, Yazd, Iran\\
{\tt zare.zakieh@yahoo.com~~~alikhani@yazd.ac.ir}\\
$^2$Department of Mathematics, College of Science, Shiraz University, Shiraz, Iran
{\tt mr\_oboudi@shirazu.ac.ir}
\end{center}

\begin{abstract}
Let $ G=(V,E) $ be a simple graph of order $ n $ and size $ m $. A connected edge cover set of a graph is a subset $S$ of edges such that every vertex of the graph is incident to at least one edge of $S$  and the subgraph induced  by $S$  is connected. 
We initiate the study of the number of the connected edge cover sets of a graph $G$ with cardinality $i$, $ e_{c}(G,i) $  and consider the generating function for $ e_{c}(G,i) $
which is called the connected edge cover polynomial of $ G $. After obtaining some results for this polynomial, we investigate this polynomial for some certain graphs.  
\end{abstract}


\medskip
\noindent{\bf AMS Subj.\ Class.}: 05C30, 05C69.  

\section{Introduction}
 Let $G=(V,E)$ be a simple graph. The {\it order} and the size of $G$ is the number of  vertices and the number  of edges of $G$, respectively.
 For every graph $G$ with no isolated vertex, an edge covering of $G$ is a set of edges of $G$ such that every vertex is incident with at least one edge of the set. In other words, an edge covering of a graph is a set
 of edges which together meet all vertices of the graph. A minimum edge covering is an edge covering of the smallest possible size. The edge covering number of $G$ is the size of a minimum edge covering of
 $G$ and is denoted by $\rho(G)$. We let $\rho(G) = 0$, if $G$ has some isolated vertices.
 For a detailed treatment of these parameters, the reader is referred to~\cite{saeid1,JAS,bond,GRo}.
  Let $\mathcal{E}(G,i)$ be the family of all edge coverings of a graph $G$ with cardinality $i$ and
 let $e(G,i)=|{\mathcal{E}}(G,i)|$.
 The { edge cover polynomial} $E(G,x)$ of $G$ is defined as
 \[
 E(G, x)=\sum_{ i=\rho(G)}^{m} e(G, i) x^{i},
 \]
  where $\rho(G)$ is the edge covering number of $G$. Also, for a graph $G$ with some isolated vertices
 we define $E(G, x) = 0$. Let $E(G, x) = 1$, when
  both order and size of $G$ are zero (see \cite{saeid1}).
 
 In  \cite{saeid1} authors have characterized all graphs whose edge cover polynomials have exactly one or two distinct roots and moreover they proved that these roots
 are contained in the set $\{-3,-2,-1, 0\}$. In \cite{JAS}, authors constructed some  infinite families of graphs whose
 edge cover polynomials have only roots $-1$ and $0$. Also, they studied the edge coverings and edge
  cover polynomials of cubic graphs of order $10$. As a consequence , they have shown  that the all cubic graphs of order $10$ (especially the
 Petersen graph) are
 determined uniquely by their edge cover polynomials.

Motivated by the edge cover number, we consider the following definition. 

\begin{definition}
	A {\it connected edge cover set} of  graph $G$ is a subset $S$ of edges such that every vertex of $G$  is incident to at least one edge of $S$  and the subgraph induced  by $S$  is connected. The connected edge cover number of $G$, $ \rho_{c}(G)$, is the minimum cardinality of the connected edge cover.
\end{definition}

Also, we state the following definition for the connected edge cover polynomial. 

\medskip

\begin{definition}	
		The {\it connected edge cover polynomial} of $ G $ is the polynomial 
	\[
	 E_{c}(G,x)=\sum_{i=1}^{m} e_{c}(G,i)x^{i},
	 \]
	  where $ e_{c}(G,i) $ is the number of connected edge cover set of size $ i $.
\end{definition}

For two graphs $G$ and $H$, the corona $G\circ H$ is the graph arising from the
disjoint union of $G$ with $| V(G) |$ copies of $H$, by adding edges between
the $i$th vertex of $G$ and all vertices of $i$th copy of $H$. The corona $G\circ K_1$, in particular, is the graph constructed from a copy of $G$, where for each vertex $v\in V(G)$, a new vertex $u$ and a pendant edge $\{v, u\}$ are added.
It is easy to see that the corona operation of two graphs does not have the commutative property.

\medskip

Usually  the generalized friendship graph is denoted by $ F_{n,m} $ which is a collection of $ n $ cycles (all of order $ m$), meeting at a common vertex.

\medskip

Two graphs $ G $ and $ H $ are said to be connected edge covering equivalent, or simply {\it ${\mathcal{E}_{c}}$-equivalent}, written $ G\sim_{c}H $, if $ E_{c}(G,x)=E_{c}(H,x) $. It is evident that the relation $\sim_{c}$ of being
${\mathcal{E}_{c}}$-equivalence is an equivalence relation on the family ${\cal G}$ of graphs, and thus ${\cal G}$ is partitioned into equivalence classes, called the {\it ${\mathcal{E}_{c}}$-equivalence classes}. Given $G\in {\cal G}$, let
\[
[G]=\{H\in {\cal G}:H\sim_{c} G\}.
\]
We call $[G]$ the equivalence class determined by $G$.

A graph $ G $ is said to be connected edge covering unique or simply {\it $ E_{c} $-unique}, if $ [G]={G} $.

\medskip

In this paper, we obtain the connected edge cover polynomial for certain graphs. 

\section{Connected edge cover polynomial}
Here, we state some new results on the connected edge cover number and the connected edge cover polynomial.  The following theorem is easy to obtain: 

\begin{theorem}

For every natural number $ n\geq 3 $,
\begin{enumerate}
	\item [(i)]
	$ E_{c}(K_{n},x)=E(K_{n},x)-\sum_{ i=\lceil n/2\rceil}^{n-2} e(K_{n}, i) x^{i} $.

	\item[(ii)]
	For every natural number $ n\geq 3 $, 
	$ \rho_{c}(C_{n})=n-1 $ and $ E_{c}(C_{n},x)=\sum_{ i=n-1}^{n} {n \choose i} x^{i} $.
	\item[(iii)] 
For every natural number $ n\geq 5 $,
$ E_{c}(P_{n},x)= x^{n-1} $.
\end{enumerate}
\end{theorem}

\medskip

\begin{theorem}
For every natural numbers $n$ and $ m\geq 3$,
$ E_{c}(F_{n,m},x)=\sum_{i=0}^{n} {n \choose i} m^{i} x^{mn-i} $.
\end{theorem}
\begin{proof}
We know that $\Delta(F_{n,m})=mn$ and $\delta(F_{m,n})=n(m-1)$.
To construct a connected edge cover set of $F_{m,n}$  with cardinal $ mn-i$, it is enough to choose $ m-1 $ edges from $ m $ edges of $i$ cycles $C_m$. So $e_c(F_{m,n},mn-i)={n \choose i} m^{i}$ and so we have the result. \qed
\end{proof}

\begin{theorem}
If $ G $ is a graph with order $ n $ and $ E_{c}(G ,x)=E_{c}(K_{n} ,x) $, then $ G=K_{n} $.
\end{theorem}

\begin{proof}
Since the degree of $ E_{c}(K_{n} ,x) $ is $m=\frac{n(n-1)}{2}$ and $ E_{c}(G ,x)=E_{c}(K_{n},x) $, so $ G $ is a graph of size $ m $. On the other hand, the only connected graph of the order $ n $ and size $ m=\frac{n(n-1)}{2}$ is graph $ K_{n} $. Therefore $ G=K_{n} $.\qed
\end{proof}

Here, we obtain an recursive formula for the connected edge cover polynomial of graphs.
Let $u\in V(G)$. By
$N_u$ we mean the set of all edges of $G$ incident with $u$.

\begin{theorem}\label{main}

Let $ G $  be a graph, $ u, v\in V(G) $ and $ uv\in E(G) $. Then
$$ E_{c}(G, x)=(x+1)E_{c}(G\setminus uv, x)+xE_{c}(G\setminus v, x)+xE_{c}(G\setminus u, x) .$$
\end{theorem}

\begin{proof}
If $G$ has an isolated vertex, then $G$ is a disconnected graph, so there is nothing to prove. Suppose that $ \delta(G)\geq1 $ and $ S $ is a connected edge covering set of $ G $ of size $ i $.
\begin{itemize}
	\item 
If $ uv\notin S $, then we have two cases:

\begin{enumerate}

\item[(1)]  $ deg(v)=1 $ or $ deg(u)=1 $. So  $ S $ is a disconnected graph.

\item[(2)]  $ deg(v)>1 $ and $ deg(u)>1 $. So $ S $ is a connected edge covering set of $ G\setminus uv $ with size $ i $.
\end{enumerate}

\item If $ uv\in S $, then we have the following cases:
\begin{enumerate}

\item[(i)] 
$ |S\cap N_{u}|=|S\cap N_{v}|=1 $. So in this case $ S $ is disconnected graph.

\item[(ii)] 
$ |S\cap N_{u}|>1 $ and $|S\cap N_{v}|=1 $. Therefore $ S\setminus uv $ is a connected edge covering set of $ G\setminus v $ with size $ i-1 $.

\item[(iii)]
  $|S\cap N_{u}|= 1 $ and $|S\cap N_{v}|>1 $. Therefore $ S\setminus uv $ is a connected edge covering set of $ G\setminus u $ with size $ i-1 $.

\item[(iv)] 
$|S\cap N_{u}|>1 $ and $|S\cap N_{v}|>1 $. Therefore $ S\setminus uv $ is a connected edge covering set of $ G\setminus uv $ with size $ i-1 $.

\end{enumerate}
\end{itemize} 
So we have
$$ e_{c}(G, i)= e_{c}(G\setminus uv, i)+ e_{c}(G\setminus v, i-1)+ e_{c}(G\setminus u, i-1)+ e_{c}(G\setminus uv, i-1), $$

and so we have the result. \qed 
\end{proof} 

\medskip

By Theorem \ref{main}, we have the following corollary:

\begin{corollary}
	\begin{enumerate} 
\item[(i)] 
For every natural number $ n\geq 3 $,
$ E_{c}(P_{n}, x)= xE_{c}(P_{n-1}, x) $.

\item[(ii)] 
For every natural number $ n\geq 4 $,
$ E_{c}(C_{n}, x)= xE_{c}(C_{n-1}, x)+x^{n-1} $.

\end{enumerate}
\end{corollary}

Here, we consider the connected edge cover number and the connected edge cover polynomial for corona of some graphs.

\begin{theorem} 
	\begin{enumerate}
		\item [(i)]
		For any connected graph $ G $ of order $ n $, $ \rho_{c}(G\circ K_{1})=2n-1$.
		
		\item[(ii)] 
		For any natural number $ n\geq3 $,   and for every $ i $, $ 2n-1\leq i\leq n+\frac{n(n-1)}{2}$,  $$ e_{c}(K_{n}\circ K_{1}, i)={\frac{n(n-1)}{2} \choose i-n}-n{n-1 \choose i-n} .$$
	\end{enumerate}

\end{theorem}
\begin{proof}
		\begin{enumerate}
			\item [(i)]
If $ S $ is a connected edge covering of $ G\circ K_{1} $, then $S$ contains  at least $ n-1 $ edges of the graph $ G $ and  $ n $ edges which connect the vertices of $G$ and the copies of graph $ K_{1} $. So we have $|S|\geq 2n-1$ and so we have the result.  

\item[(ii)] 
Any edge cover set of   $ K_{n}\circ K_{1} $ of size $ i $ should contain $n$ edges of the outer $C_n$. Now we should choose $i-n$ edges from any $n$ induced subgraph $K_{n-1}$. Therefore, we have the result. \qed
\end{enumerate} 
 \end{proof}

\medskip

\begin{theorem} 
Let $ G $ be a connected graph of order $ n $ and size $ m $. If $ E_{c}(G,x)=\sum_{i=1}^{m} e_{c}(G,i)x^{i} $, then the following hold:
\begin{enumerate}
\item[(i)] $ E_{c}(G, x) $ is a monic polynomial of degree $ m $.

\item[(ii)] $ n\leq \rho_{c}(G)+1 $.

\item[(iii)] For $ i\geq m-\delta+1 $, $ e_{c}(G, i)={m \choose i} $. Moreover, if $ i_{0}=min \lbrace i \vert e_{c}(G, i)={m \choose i}\rbrace $, then $ \delta=m-i_{0}+1 $.
\end{enumerate}
\end{theorem}

\begin{proof}
\begin{enumerate}

\item[(i)] Since $ E(G) $ is the unique connected edge covering of $ G $ of size $ m $, so the result follows. 

\item[(ii)] Since any $ n-1 $ edges in graph $G$ is a connected edge covering of $ G $, so we have the result.

\item[(iii)] Let $ i\geq m-\delta+1 $. So every subset $S\subseteq E(G)$ of size $i$ is a connected  edge covering of $G$.
  Now, suppose that $i \leq m-\delta$. Consider a vertex $v$ of degree $\delta$. Let $A\subseteq \overline{N_v}$, such
that $|A|=i$. Clearly, $A$ is not a connected edge covering of $G$. So $e_c(G,i)<{m\choose i}$. \qed

\end{enumerate}
\end{proof}

\medskip

\begin{corollary} 
Let $ G $ and $ H $ be two connected graphs of size $ m_{1} $ and $ m_{2} $. If $ E_{c}(H, x)=E_{c}(G, x) $, then $ \rho_{c}(G)=\rho_{c}(H) $, $ m_{1}=m_{2} $ and $ \delta(G)=\delta(H) $.

\end{corollary}

\medskip

\section{Cubic graphs of order $6$, $8$ and the Petersen graph}
In this section, we compute the number of connected edge cover set of size $ \rho_{c} $ for cubic graphs of order $6$, $8$ and the Petersen graph.  Domination polynomials of cubic graphs of order $10$ has studied in \cite{turk} and the Coalition of cubic graphs of order at most $10$ studied in \cite{CCO}.    
The cubic graphs of order $6$ has shown in Figure \ref{1}. 
\medskip

\begin{figure}[h!]
	\centering
	\includegraphics[scale=0.8]{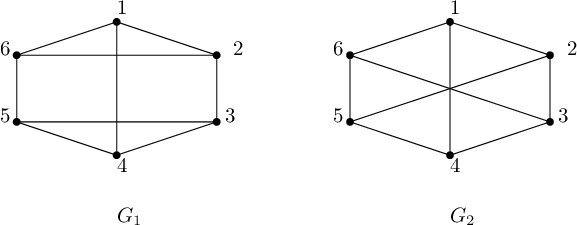}
	\caption{Cubic graphs of order 6} \label{1} 
\end{figure}

The following results give $e_c(G_1, \rho_{c}(G_1))$ and $e_c(G_2, \rho_{c}(G_2))$ for the cubic graphs of order $6$. 

\begin{theorem} \label{cub6}
		 $ e_{c}(G_{1},5)= e_{c}(G_{2}, 5)=81$.
\end{theorem}
\begin{proof}
		Consider the graphs $G_1$ and $G_2$ in Figure \ref{1}. To construct a connected edge covering set $S$ of size $5$:

\noindent $\bullet$
 Choose $5$ edges from the cycle 
$ \{ \{ 1,2 \},\{ 2,3 \},\{ 3,4 \},\{ 4,5 \},\{ 5,6 \},\{ 6,1\} \}$ in Figure \ref{1}. So we have $6$ distinct sets. 

\noindent $\bullet$
 Choose $4$ edges from the cycle $ \{ \{ 1,2 \},\{ 2,3 \},\{ 3,4 \},\{ 4,5 \},\{ 5,6 \},\{ 6,1\} \} $ and one another edge that one of its end-vertex is a vertex which is not on the $4$ chosen edges. So we have $ {6 \choose 4}{1 \choose 1}=15 $ distinct connected edge covering set.

\noindent $\bullet$
 Choose $3$ edges from the cycle $ \{ \{ 1,2 \},\{ 2,3 \},\{ 3,4 \},\{ 4,5 \},\{ 5,6 \},\{ 6,1\} \} $ and $2$ edges from $ \{ \{ 1,4 \}, \{ 2,6 \}, \{ 3,5 \} \} $, except for the case that $3$ edges of the cycle $ \{ \{ 1,2 \}, \{ 2,3 \},\\ \{ 3,4 \},\{ 4,5 \},\{ 5,6 \},\{ 6,1 \} \} $ are connected.  So in case, we have 
$ {6 \choose 3}{3 \choose 2}-{6 \choose 1}\times2=48 $  distinct connected edge covering set.

\noindent 
$\bullet$
  Choose $3$ edges from  $ \{ \{ 1,4 \}, \{ 2,6 \}, \{ 3,5 \}\} $ and $2$ edges from $ \{ \{ 1,2 \},\{ 2,3 \},\{ 3, \\ 4 \},\{ 4,5 \},\{ 5,6 \},\{ 6,1\} \} $, except for three  states $ \{ \{\{1,2\},\{6,1\}\}, \{\{2,3\},\{5,6\}\}, \{\{3,4\},\\\{4  ,5\}\} \} $. So in case  we have
 $ {3 \choose 3}\times [{6 \choose 2}-3]=12 $ distinct connected edge covering set.

Therefore, by the addition principle, $e_{c}(G_{1},5)=81$. \qed

\end{proof}

Similar to the proof of Theorem \ref{cub6}, we can compute another coefficients of cubic graphs of order $6$ and we have the following result:  

\begin{theorem}
If $G_1$ and $G_2$ are two cubic graphs of order $6$ (Figure \ref{1}), then 

$$ E_{c}(G_{1}, x)=E_{c}(G_{2}, x)=x^{9}+{9 \choose 8}x^{8}+{9 \choose 7}x^{7}+{9 \choose 6}x^{6}+81x^{9}.$$
\end{theorem}

\begin{figure}[ht]
	\centering
	\includegraphics[scale=0.8]{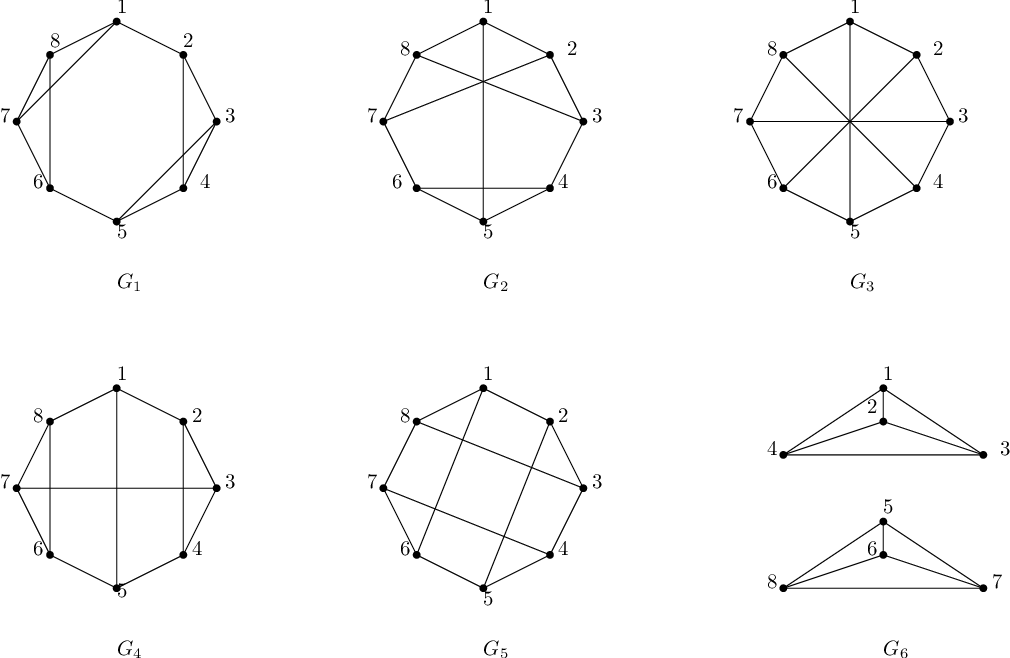}
	\caption{Cubic graphs of order 8} \label{2} 
\end{figure} 

Here, we obtain the number of connected edge covering sets  of size $\rho_c$ of cubic graphs of order $8$ which have shown in Figure \ref{2}. 

\begin{theorem}\label{cube8}
	\begin{enumerate}
		\item[(i)] $ e_{c}(G_{1},7)=324$.
		\item[(ii)]
		$ e_{c}(G_{2}, 7)=338 $.
		\item[(iii)] 
		$ e_{c}(G_{3}, 7)= e_{c}(G_{4}, 7)=332 $.
		\item[(iv)]
		 $ e_{c}(G_{5}, 7)=344 $.
		\end{enumerate}
 \end{theorem}
\begin{proof}
The proof of all parts are easy and similar. For instance, we state the proof of Part 
(i). To construct a connected edge covering of size $7$ of $G_1$, we have some cases:
\begin{enumerate}

\item[(1)] Choose seven edges from $ \{ \{1,2\}, \{2,3\}, \{3,4\}, \{4,5\}, \{5,6\}, \{6,7\}, \{7,8\}, \{8,1\} \} $. So  we have ${8 \choose 7}=8$ distinct connected edge covering sets.

\item[(2)] Choose six edges from $ \{ \{1,2\}, \{2,3\}, \{3,4\}, \{4,5\}, \{5,6\}, \{6,7\}, \{7,8\}, \{8,1\} \} $  and one another edge that one of its end-vertices  is a vertex which is not on the $6$ chosen edges. So we have $ {8 \choose 6}{1 \choose 1}=28 $ distinct connected edge covering sets.

\item[(3)] Choose five edges from $ \{ \{1,2\}, \{2,3\}, \{3,4\}, \{4,5\}, \{5,6\}, \{6,7\}, \{7,8\}, \{8,1\} \} $. We have the following subcases:

\textbf{Case 1}. Choose five edges of induced subgraph $ P_{6} $ from the cycle $ \{ \{1,2\}, \{2,3\},\\ \{3,4\}, \{4,5\}, \{5,6\}, \{6,7\}, \{7,8\}, \{8,1\} \} $, and choose $2$ edges from $2$ edges that are connected to vertices which are not on the $5$ chosen edges.  So, we have $ 8\times {2 \choose 2}=8 $ distinct connected edge covering sets.

\textbf{Case 2}. Choose five edges from outer $C_8$ such that a vertex of $C_8$, say $v$ is not incident to chosen edges. Then we can choose two edges of $ \{ \{1,2\}, \{2,3\},\\ \{3,4\}, \{4,5\}, \{5,6\}, \{6,7\}, \{7,8\}, \{8,1\} \} $ such that the distance between an end-vertex and the vertex $v$ is two. So,  we choose one of the edges with end-vertex $v$ and one of the edges whose induced subgraph have leaf.  Therefore in this case there are
$ {8 \choose 1}{1 \choose 1}{4 \choose 3}{1 \choose 1}=32 $ connected edge cover sets.
 
\textbf{Case 3}. If the induced subgraph of $5$ chosen edges does not have an isolated vertex, then we subtract two previous cases from the total states and we know that two pairs vertices from $6$ vertices with degree one have repetitive edges connecting to each other, so we have $4$ vertices and we need to choose $2$ edges connected to them. Thus in this case there are
$ [{8 \choose 5}-(8+32)]\times {4 \choose 2}=96 $ connected edge cover sets.

\item[(4)] Choose four edges from $ \{ \{1,2\}, \{2,3\}, \{3,4\}, \{4,5\}, \{5,6\}, \{6,7\}, \{7,8\}, \{8,1\} \} $. We have the following subcases:

\textbf{Case 1}. Consider an isolated vertex (in the induced subgraph), say $v$. Choose one of the edges with end-vertex $v$, one edge from edges that are adjacent with an edge with end-vertex $v$  and one of the edges from two edges that is adjacent to an edge $K_2$ in induced subgraph. So in this case, there are 
$ {8 \choose 1}{1 \choose 1}{1 \choose 1}{2 \choose 1}=16 $ connected edge cover sets.

\textbf{Case 2}. Choose two  isolated vertices (in the induced subgraph), two vertices from 2 vertices that are connected to isolated vertices and one of edge from 2 other edges. Thus in this case there are
$ {8 \choose 2}{2 \choose 2}{2 \choose 1}=56 $ connected edge cover sets.

\textbf{Case 3}. Choose there  isolated vertices (in the induced subgraph) and there  vertices from 3 vertices that are connected to isolated vertices. Therefore in this case there are 
$ {8 \choose 3}{3 \choose 3}=56 $ connected edge cover sets.

\item[(5)] Choose four edges from $\{ \{1,7\}, \{2,4\}, \{3,5\}, \{6,8\} \} $. We have the following subcases:

\textbf{Case 1}. Choose two edges from $ \{ \{1,2\}, \{5,6\} \} $ and one edge from other 6 edges. So in this case there are 
$ {4 \choose 4}{2 \choose 2}{6 \choose 1}=6 $ connected edge cover sets.

\textbf{Case 2}. Choose one edge from $ \{ \{1,2\}, \{5,6\} \} $, one edge from $ \{ \{2,3\}, \{3,4\}, \{4,5\} \} $ and one edge from $ \{ \{6,7\}, \{7,8\}, \{8,1\} \} $. So in this case there are
$ {4 \choose 4}{2 \choose 1}{3 \choose 1}{3 \choose 1}=18 $ connected edge cover sets.

According the addition principle, $e_{c}(G_{1}, 7)=324 $.  \qed

\end{enumerate}
\end{proof}  

Similar to the proof of Theorem \ref{cube8}, we can obtain another coefficients of the connected edge cover polynomial of cubic graphs of order $8$ and so we have the following result:  

\begin{theorem}
If $ G_{1},...,G_{5}$ are the cubic graphs of oder $8$ (Figure \ref{2}), then 

\item[(i)] $ E_{c}(G_{1}, x)=x^{12}+{12 \choose 11}x^{11}+[{12 \choose 7}-1]x^{10}+[{12 \choose 9}-6]x^{9}+[{12 \choose 8}-6]x^{8}+324x^{7} $,

\item[(ii)] $ E_{c}(G_{2}, x)=x^{12}+{12 \choose 11}x^{11}+{12 \choose 7}x^{10}+{12 \choose 9}x^{9}+[{12 \choose 8}-2]x^{8}+338x^{7} $,

\item[(iii)] $ E_{c}(G_{3}, x)=E_{c}(G_{4}, x)=x^{12}+{12 \choose 11}x^{11}+{12 \choose 7}x^{10}+{12 \choose 9}x^{9}+[{12 \choose 8}-4]x^{8}+332x^{7} $,

\item[(iv)] $ E_{c}(G_{5}, x)=x^{12}+{12 \choose 11}x^{11}+{12 \choose 7}x^{10}+{12 \choose 9}x^{9}+{12 \choose 8}x^{8}+344x^{7} $.

\end{theorem}

\begin{figure}[h!]
	\centering
	\includegraphics[scale=1]{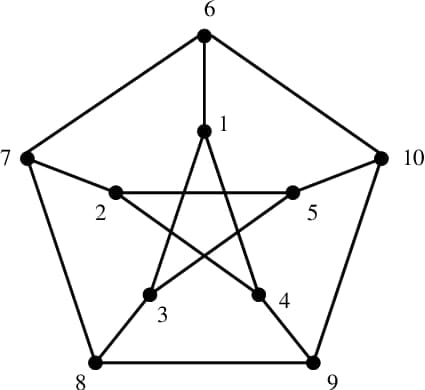}
	\caption{The Petersen graph} \label{Pet} 
\end{figure}

\medskip

There are exactly $21$ cubic graphs of order $10$ (\cite{JAS}), of them all, the Petersen graph $P$ is the most well-known. Here, we compute the number of connected edge cover sets of $P$ of size $9$.  

\begin{lemma} 
 If $P$ is the Petersen graph, then $ e_{c}(P, 9)=235$.

\end{lemma}

\begin{proof}
To construct a connected edge covering of $P$ of  size $9$ (Figure \ref{Pet}), we have some cases:
\begin{enumerate}

\item[(i)] Choose four edges from the cycle $ \{ \{1,2\},\{2,3\},\{3,4\},\{4,5\},\{5,1\} \} $, one edge from $ \{ \{1,6\}, \{2,7\}, \{3,8\}, \{4,9\}, \{5,10\} \} $ and four edges from $ \{ \{6,8\}, \{6,9\}, \{7,9\}, \\ \{7,10\}, \{8,10\} \} $. So in case,  there are $ {5 \choose 4}{5 \choose 1}{5 \choose 4}=125 $  connected edge cover sets.

\item[(ii)] Choose four edges from the cycle $ \{ \{1,2\},\{2,3\},\{3,4\},\{4,5\},\{5,1\} \} $, four edges from $ \{ \{1,6\}, \{2,7\}, \{3,8\}, \{4,9\}, \{5,10\} \} $ and there are four ways to connect an isolated vertex to other vertices. Therefore in case, there are $  {5 \choose 4}{5 \choose 4}\times 4=100 $  connected edge cover sets.

\item[(iii)] Choose five edges from $ \{ \{1,6\}, \{2,7\}, \{3,8\}, \{4,9\}, \{5,10\} \} $ and four edges from $ \{ \{6,8\}, \{6,9\}, \{7,9\}, \{7,10\}, \{8,10\} \} $. So in this
case there are $ {5 \choose 5}{5 \choose 4}=5 $  connected edge cover sets.

\item[(iv)] Choose four edges from the cycle $ \{ \{1,2\},\{2,3\},\{3,4\},\{4,5\},\{5,1\} \} $ and five edges from $ \{ \{1,6\}, \{2,7\}, \{3,8\}, \{4,9\}, \{5,10\} \} $. So in this
case there are $ {5 \choose 4}{5 \choose 5}=5 $  connected edge cover sets.

\end{enumerate}

Therefore, by the addition principle, we have the result.  \qed 

\end{proof}

\section{Conclusion} 
In this paper we considered the connected edge cover sets of a graph and initiated the study of the number of the connected edge cover sets with cardinality $i$. We introduced the  connected edge cover polynomial of a graph and found some of its properties. We computed this polynomial for cubic graphs of order $6,8$ and also we computed the number of connected edge cover sets of the Petersen graphs of size $9$. There are many open problems for future works. We state some of them:

\begin{enumerate}
	\item [(i)]
	Study the roots of the connected edge cover polynomial of a graph. 
	
	\item[(ii)] 
	What are the formula for the connected edge cover polynomial of some graph products, such as, Cartesian, corona, join and lexicographic. 
	
	\item[(iii)] 
	Examine the effects on $\rho(G), \rho_c(G)$ and $E_c(G,x)$ when $G$ is modified by operations 	on vertex and edge of $G$. 
	  
\end{enumerate}

\end{document}